\documentclass[12pt,a4paper]{article}

\usepackage{authblk}
\usepackage[margin=2cm]{geometry}
\usepackage{t1enc}
\usepackage[utf8]{inputenc}
\usepackage{amsthm,amsmath,amssymb}
\usepackage{graphicx}
\usepackage{enumerate}
\usepackage{hyperref}
\usepackage{bm}
\usepackage{comment}
\usepackage{amsfonts}
\usepackage{graphicx,caption}
\usepackage{bm}
\usepackage{amsmath, amsthm, amssymb}
\usepackage{graphicx}
\usepackage{hyperref}
\usepackage{relsize}
\usepackage{blkarray}
\usepackage{algpseudocode}
\newcommand{\lC}{L}

\usepackage{bbm}

\theoremstyle{plain}
\usepackage{amsthm}
\makeatletter
\newcommand{\newreptheorem}[2]{\newtheorem*{rep@#1}{\rep@title}\newenvironment{rep#1}[1]{\def\rep@title{#2 \ref*{##1}}\begin{rep@#1}}{\end{rep@#1}}}
\makeatother

\newtheorem{theorem}{Theorem}
\newtheorem*{theorem-non}{Theorem}
\newtheorem*{non-lemma}{Lemma}
\newtheorem{lemma}[theorem]{Lemma}
\newreptheorem{lemma}{Lemma}

\newtheorem{claim}[theorem]{Claim}

\newtheorem{conjecture}[theorem]{Conjecture}
\newtheorem{problem}[theorem]{Problem}

\theoremstyle{definition}

\DeclareMathOperator{\Sub}{Sub}

\DeclareMathOperator{\RowS}{RowSpace}

\DeclareMathOperator{\cok}{cok}

\DeclareMathOperator{\supp}{supp}
\DeclareMathOperator{\Sur}{Sur}
\DeclareMathOperator{\Hom}{Hom}
\DeclareMathOperator{\Aut}{Aut}

\DeclareMathOperator{\Sg}{Sub}

\begin{document}
\title{A phase transition for the cokernels of random band matrices over the p-adic integers}
\author{Andr\'as M\'esz\'aros}
\date{}
\affil{HUN-REN Alfr\'ed R\'enyi Institute of Mathematics}
\maketitle
\begin{abstract}
Let $\mathbf{B}_n$ be an $n\times n$ Haar-uniform band matrix over $\mathbb{Z}_p$ with band width $w_n$. We prove that $\cok(\mathbf{B}_n)$ has Cohen-Lenstra limiting distribution if and only if \[\lim_{n\to\infty} \left(w_n-\log_p(n)\right)=+\infty.\]

\end{abstract}

\section{Introduction}

Let $p$ be a prime. Given two positive integers $n$ and $w$, let $\mathcal{B}_{n,w}$ be the set of $n\times n$ band matrices over the $p$-adic integers $\mathbb{Z}_p$ with band width $w$, that is, let 
\[\mathcal{B}_{n,w}=\{B\in M_n(\mathbb{Z}_p)\,:\,B(i,j)=0\text{ for all }i,j\text{ such that }|i-j|>w\}.\]

Let $w_n$ be a sequence of positive integers, and let $\mathbf{B}_n$ be a Haar uniform element of $\mathcal{B}_{n,w_n}$. The cokernel of $\mathbf{B}_n$ is the random abelian $p$-group
\[\cok(\mathbf{B}_n)=\mathbb{Z}_p^n/\RowS(\mathbf{B}_n),\]
where $\RowS(\mathbf{B}_n)$ is the $\mathbb{Z}_p$-submodule of $\mathbb{Z}_p^n$ generated by the rows of $\mathbf{B}_n$. 

Our main theorem states that $\cok(\mathbf{B}_n)$ exhibits a phase transition. If $w_n$ grows fast enough, then  $\cok(\mathbf{B}_n)$ has Cohen-Lenstra limiting distribution. If $w_n$ does not grow fast enough, then $\cok(\mathbf{B}_n)$ does not have Cohen-Lenstra limiting distribution. We determine precisely where this phase transition happens.

\begin{theorem}\hfill\label{thmmain}
\begin{enumerate}[(a)]
\item If $\lim_{n\to\infty} \left(w_n-\log_p(n)\right)=+\infty$, then $\cok(\mathbf{B}_n)$ has Cohen-Lenstra limiting distribution, that is, for all finite abelian $p$-groups $G$, we have
\begin{equation}\label{CohenLenstra}\lim_{n\to\infty} \mathbb{P}(\cok(\mathbf{B}_n)\cong G)=\frac{1}{|\Aut(G)|}\prod_{j=1}^{\infty}\left(1-p^{-j}\right).\end{equation}
\item If $w_n-\log_p(n)$ does not converge to $+\infty$, then $\cok(\mathbf{B}_n)$ does not have Cohen-Lenstra limiting distribution.
\end{enumerate}
\end{theorem}

The distribution on the right hand side of \eqref{CohenLenstra} is called the Cohen-Lenstra distribution. It first appeared in a conjecture about the distribution of class groups of quadratic number fields~\cite{cohen2006heuristics}. We prove the first part of Theorem~\ref{thmmain} by using the moment method of Wood~\cite{wood2017distribution}. By now the moment method is considered the standard method of establishing the Cohen-Lenstra limiting distribution of cokernels of random matrices. See \cite{wood2019random,nguyen2022random,recent1,recent2,recent3,recent4,recent5,recent6,recent7,recent8,recent9,recent10,meszaros2020distribution,meszaros2023cohen} for some recent examples. See also \cite{sawin2022moment,van2024symmetric,wood2022probability} for further results on the moment problem. With the choice of $w_n=n$, Theorem~\ref{thmmain} gives back the classical result of Friedman and Washington~\cite{friedman1987distribution} on the cokernels of Haar uniform matrices over $\mathbb{Z}_p$. Theorem~\ref{thmmain} provides us random matrices with only $O(n\log n)$ non-zero entries and Cohen-Lenstra limiting behaviour. Note that the author provided random matrices with $O(n)$ non-zero entries and Cohen-Lenstra limiting behaviour~\cite{meszaros2020distribution,meszaros2023cohen}. Excluding these examples, all other results mentioned above are about denser random matrices. 

The second part of Theorem~\ref{thmmain} is proved by showing that for small band widths, the kernel of the mod $p$ reduction of $\mathbf{B}_n$ contains several localized vectors with sufficiently large probability, see Theorem~\ref{thmloc}. This phenomenon can be viewed as an analogue of the localization/delocalization phase transition of Gaussian band matrices as we explain in Section~\ref{secGaussian}. 

\subsection{Band matrices over $\mathbb{F}_p$}

Let $\overline{\mathbf{B}}_n\in M_n(\mathbb{F}_p)$ be the mod $p$ reduction of $\mathbf{B}_n$. In other words, $\overline{\mathbf{B}}_n$ is a uniform random element of the set 
\[\{B\in M_n(\mathbb{F}_p)\,:\,B(i,j)=0\text{ for all }i,j\text{ such that }|i-j|>w_n\}.\]

The first part of Theorem~\ref{thmmain} has the following corollary, see also Lemma~\ref{CLrank}.

\begin{theorem}\label{thmrank}
Assume that $\lim_{n\to\infty} \left(w_n-\log_p(n)\right)=+\infty$. Then for all nonnegative integers $k$, we have
\[\lim_{n\to\infty} \mathbb{P}(\dim \ker \overline{\mathbf{B}}_n=k)=p^{-k^2} \prod_{j=1}^{k} \left(1-p^{-j}\right)^{-2} \prod_{j=1}^{\infty}\left(1-p^{-j}\right).\]
\end{theorem}

Next we consider the case when $w_n-\log_p(n)$ does not converge to $+\infty$.

Given two positive integers $n$ and $\lC $, we set $n_0=\left\lfloor\frac{n}\lC \right\rfloor$, and for $i=1,2,\dots,\lC$, we define
\[V_i=\left\{v\in \mathbb{F}_p^n\,:\,\emptyset\neq \supp v\subset \left[(i-1)n_0+1\,,\,i n_0\right]\right\}.\]

We say that $\overline{\mathbf{B}}_n$ has an $\lC $-localized kernel if the following event occurs:
\[\text{For all }i=1,2,\dots,\lC, \text{ there is a }v_i\in V_i\text{ such that }\overline{\mathbf{B}}_n v_i=0.\]

Observe that for $i\neq j$ and $v_i\in V_i,v_j\in V_j$, the vectors $v_i$ and $v_j$ do not interact in the sense that the events  $\overline{\mathbf{B}}_n v_i=0$ and $\overline{\mathbf{B}}_n v_j=0$ depend on disjoint subsets of the entries of~$\overline{\mathbf{B}}_n$. In particular, these two events are independent.  

Note that if $\overline{\mathbf{B}}_n$ has an $\lC $-localized kernel, then $\dim \ker \overline{\mathbf{B}}_n\ge \lC $.

\begin{theorem}\label{thmloc}
Assume that $ w_n-\log_p(n)$ does not converge to $+\infty$. Then there is a $\kappa>0$ with the property that for all positive integers $\lC $, we have infinitely many $n$ such that
\[\mathbb{P}(\,\overline{\mathbf{B}}_n\text{ has an $\lC $-localized kernel}\,)\ge \left(\frac{\kappa}\lC \right)^{4\lC }.\]
Consequently, for all positive integers $\lC $, we have infinitely many $n$ such that
\begin{equation}\label{heavytail}\mathbb{P}(\dim\ker \overline{\mathbf{B}}_n\ge \lC )\ge \left(\frac{\kappa}\lC \right)^{4\lC }.
\end{equation}
\end{theorem}

Note that \eqref{heavytail} excludes the possibility of a Cohen-Lenstra limiting distribution for $\cok(\mathbf{B}_n)$, so the second part of Theorem~\ref{thmmain} follows. Indeed, if $\cok(\mathbf{B}_n)$ had Cohen-Lenstra limiting distribution that would imply that for all positive integers $\lC$, we have
\[\lim_{n\to\infty}\mathbb{P}(\dim\ker \overline{\mathbf{B}}_n\ge \lC )\le C p^{-\lC^2},\]
where $C$ is a constant depending only on $p$, see Lemma~\ref{CLrank}. Choosing $\lC$ large enough, this would contradict \eqref{heavytail}.

When $w_n$ grows fast enough, we do not see localized vectors in $\ker \overline{\mathbf{B}}_n$. In fact, $\ker \overline{\mathbf{B}}_n$ only consists of typical vectors in the following sense.

\begin{theorem}\label{thmdeloc}
Assume that $\lim_{n\to\infty} \left(w_n-\log_p(n)\right)=+\infty$. Let $\mathcal{E}_n$ be a deterministic subset of  $\mathbb{F}_p^n\setminus\{0\}$ such that $\lim_{n\to\infty} \frac{|\mathcal{E}_n|}{p^n}=0$. Then
\[\lim_{n\to\infty}\mathbb{P}(\ker \overline{\mathbf{B}}_n\cap \mathcal{E}_n\neq\emptyset)=0.\]
\end{theorem}

We prove two further phase transitions.

\begin{theorem}\label{thmtight}
The sequence of random variables $\dim\ker \overline{\mathbf{B}}_n$ is tight if and only if \[\liminf_{n\to\infty} \left(w_n-\log_p(n)\right)>-\infty.\]
\end{theorem}

 It is also very natural to normalize  $\dim\ker \overline{\mathbf{B}}_n$ by $n$. The next theorem gives a phase transition for this normalized dimension.

\begin{theorem}\label{thmnormalized}
 The sequence of random variables \[\frac{\dim\ker \overline{\mathbf{B}}_n}n\]
 converges to $0$ in probability if and only if $\lim_{n\to\infty} w_n=\infty.$
 \end{theorem}
\subsection{Comparison with Gaussian band matrices }\label{secGaussian}
An $n\times n$ Gaussian band matrix with band width $w$ is a random matrix of the form,
\[X_{n,w}=\frac{1}{\sqrt{w}}\begin{pmatrix}
d_1&a_{1,2}&a_{1,3}&\cdots&\cdots& a_{1,w+1}&0&\cdots&\cdots&0\\
\overline{a_{1,2}}&d_2&a_{2,3}&a_{2,4}&\cdots&\cdots&a_{2,w+2}&0&\cdots& 0\\
\overline{a_{1,3}}&\overline{a_{2,3}}&\ddots&\ddots& & & & & & \vdots\\
\vdots&\overline{a_{2,4}}&\ddots& \ddots& \ddots& & & & & \vdots\\
\vdots&\vdots& &\ddots&\ddots&\ddots& & &  &\vdots\\
\overline{a_{1,w+1}}&\vdots& & &\ddots&\ddots& \ddots& &  &\vdots\\
0&\overline{a_{2,w+2}}& & & &\ddots &\ddots &\ddots &&\vdots\\
\vdots&0& & & & &\ddots &\ddots &\ddots &\vdots\\
\vdots&\vdots& & & &&  &\ddots &\ddots &\vdots\\
0&0& \cdots&\cdots &\cdots &\cdots &\cdots &\cdots &\cdots&\\

\end{pmatrix},\]
where $(d_i)_{1\le i\le n}$ and $(a_{ij})_{1\le i\le n,\,\,  i<j\le\min(i+w,n)}$ are independent standard real and complex Gaussian variables, respectively.

It is conjectured that band matrices exhibit a phase transition \cite{shcherbina2021universality}:
\begin{itemize}
    \item If $w\ll \sqrt{n}$, then the eigenvectors of $X_{n,w}$ are localized, and the local eigenvalue statistics converge to a Poisson process. 
    \item If $w\gg \sqrt{n}$, then the eigenvectors of $X_{n,w}$ are delocalized, and the local eigenvalue statistics converge to the $\text{Sine}_2$ process. 
\end{itemize}

The localization of eigenvectors was proved for $w\ll n^{1/8}$ in~\cite{schenker2009eigenvector}. Later, the exponent $1/8$ was improved to ${1/7}$ and then to ${1/4}$ \cite{peled2019wegner,cipolloni2024dynamical,chen2022random}. Finally, the conjectured exponent of $1/2$ was achieved very recently by Drogin~\cite{drogin2025localization}. For fixed $w$, Poisson eigenvalue statistics were proved in~\cite{brodie2022density}, see also~\cite{hislop2022local} for further progress towards Poisson statistics. Improving earlier results \cite{band1,band2,band3,band4,band5,bb1,bb2,bb3}, the delocalization of eigenvectors and the GUE statistics were proved for $w\gg n^{1/2}$ by Yau, Yin \cite{yau2025delocalization} and by Erd\H{o}s, Riabov~\cite{erdHos2025zigzag}.


Recent results \cite{van2023local,van2023p,recent5,van2021limits,neretin2013hua,bufetov2017ergodic,assiotis2022infinite} seem to indicate that there is some analogy between the spectrum of classical random matrix ensembles and the cokernels of random matrices over the $p$-adic integers. Theorem~\ref{thmmain} gives a $p$-adic analogue of the Poisson/$\text{Sine}_2$ phase transition. Theorems~\ref{thmloc}~and~\ref{thmdeloc} give a $p$-adic analogue of the localization/delocalization phase transition.

\medskip

\textbf{The structure of the paper:} In Section~\ref{sec2}, we discuss the moment problem for abelian $p$-groups, and explain that in order to prove the first part of Theorem~\ref{thmmain}, it is enough to show that all the surjective moments of $\cok(\mathbf{B}_n)$ converge to $1$. We also show that the first part of Theorem~\ref{thmmain} implies Theorem~\ref{thmrank}, and that Theorem~\ref{thmloc} implies the second part of Theorem~\ref{thmmain}. In Section~\ref{secmomentcalc}, we show that assuming that $\lim_{n\to\infty} \left(w_n-\log_p(n)\right)=+\infty$ all the surjective moments of $\cok(\mathbf{B}_n)$ converge to $1$, which completes the proof of the first part of Theorem~\ref{thmmain}. We finish Section~\ref{secmomentcalc} by proving  Theorem~\ref{thmdeloc}. Theorems~\ref{thmloc},~\ref{thmtight} and~\ref{thmnormalized} are proved in Sections~\ref{secloc},~\ref{sectight} and~\ref{secnorm}, respectively. In Section~\ref{secopen}, we discuss several open problems. 

\medskip

\textbf{Acknowledgement:} The author is grateful to the anonymous referees for their useful comments and suggestions.
The author was supported by the KKP 139502 project and the Marie Sk\l{}odowska-Curie Postdoctoral Fellowship "RaCoCoLe".

\section{From convergence of moments to convergence of distributions}\label{sec2}

Given a finitely generated $\mathbb{Z}_p$-module $G$ and a finite abelian group $H$, let $\Hom(G,H)$ and $\Sur(G,H)$ be the set of homomorphisms and surjective homomorphisms from $G$ to $H$, respectively. The next theorem was first proved in \cite{ellenberg2016homological}, and then in greater generality in \cite{wood2017distribution}, see also~\cite[Theorem 3.1]{wood2019random}.

\begin{theorem}\label{thmWood}
    Let $\Gamma_n$ be a sequence of random finitely generated $\mathbb{Z}_p$-modules such that for any deterministic finite abelian $p$-group $G$, we have
    \[\lim_{n\to\infty} \mathbb{E}|\Sur(\Gamma_n,G)|=1.\]
    Then for any finite abelian $p$-group $G$, we have
    \[\lim_{n\to\infty} \mathbb{P}(\Gamma_n\cong G)=\frac{1}{|\Aut(G)|}\prod_{j=1}^{\infty}\left(1-p^{-j}\right).\]

\end{theorem}

The expectation $\mathbb{E}|\Sur(\Gamma_n,G)|$ is called the surjective $G$-moment of $\Gamma_n$. Our next lemma gives an expression for $\mathbb{E}|\Sur(\Gamma_n,G)|$, when $\Gamma_n=\cok(M_n)$ for some $n\times n$ random matrix $M_n$ over~$\mathbb{Z}_p$. Let $G$ be a finite abelian $p$-group. Note that $G$ is a $\mathbb{Z}_p$-module. So for all $\mathbf{g}\in G^n$, we can define $M_n \mathbf{g}\in G^n$ in the usual way. The map $\mathbf{g}\mapsto M_n\mathbf{g}$ is a $\mathbb{Z}_p$-module homomorphism from $G^n$ to $G^n$. For the proof of the next lemma, see for example \cite[Proposition 5.2]{meszaros2020distribution}. 

\begin{lemma}\label{lemma6}
We have
\begin{align*}\mathbb{E}|\Hom(\cok(M_n),G)|&=\mathbb{E}|\{\mathbf{g}\in G^n\,:\,M_n\mathbf{g}=0\}|,\text{ and}\\\mathbb{E}|\Sur(\cok(M_n),G)|&=\mathbb{E}|\{\mathbf{g}\in G^n\,:\,M_n\mathbf{g}=0\text{ and the components of $\mathbf{g}$ generate }G\}|.
\end{align*}
\end{lemma}

For two integers $c\le d$, by $[c,d]$ we mean the set $\{c,c+1,\dots,d\}$, and $[c,c-1]$ is defined to be the empty set.  

For a positive integer $w$, an interval $[c,d]$ and a boundary condition 
\[\mathbf{b}:[c-w,c-1]\cup [d+1,d+w]\to G,\] let
\[G^{[c,d]}_\mathbf{b}=\{\mathbf{g}:[c-w,d+w]\to G\,:\,\mathbf{g}(k)=\mathbf{b}(k)\text{ for all }k\in [c-w,c-1]\cup [d+1,d+w]\}.\]

Let $\Sg(G)$ be the set of subgroups of $G$. For any $\mathbf{g}\in G^{[c,d]}_\mathbf{b}$, we define $\mathbf{H}_\mathbf{g}:[c,d]\to\Sg(G)$ by setting
\[\mathbf{H}_\mathbf{g}(k)=\left\langle \mathbf{g}(j)\,:\, k-w\le j\le k+w\right\rangle\]
for all $k\in [c,d]$. 

Given $w=w_n$, let ${0}_n$ be the all zero map defined on $[1-w,0]\cup[n+1,n+w]$. 

For $\mathbf{g}\in G^{[1,n]}_{{0}_n}$, let
\[\langle \mathbf{g}\rangle=\langle \mathbf{g}(k)\,:\,k\in [1,n]\rangle.\]

\begin{lemma}\label{lemmamoment}
We have
\begin{align*}\mathbb{E}|\Hom(\cok(\mathbf{B}_n),G)|&=\sum_{\mathbf{g}\in G^{[1,n]}_{{0}_n}} \prod_{k=1}^n |\mathbf{H}
    _{\mathbf{g}}(k)|^{-1}\text{ and }\\\mathbb{E}|\Sur(\cok(\mathbf{B}_n),G)|&=\sum_{\substack{\mathbf{g}\in G^{[1,n]}_{{0}_n}\\\langle\mathbf{g}\rangle=G}} \prod_{k=1}^n |\mathbf{H}
    _{\mathbf{g}}(k)|^{-1}.
    \end{align*}

\end{lemma}
\begin{proof}
Let $\tilde{\mathbf{g}}$ be a fixed deterministic element of $G^{[1,n]}$. By appending $\tilde{\mathbf{g}}$ with zeros, we obtain a $\mathbf{g}\in G^{[1,n]}_{{0}_n}$.

Consider the homomorphism $\varphi:\mathcal{B}_{n,w_n}\to G^n$ defined by $\varphi(B)=B\tilde{\mathbf{g}}$. The image of $
\varphi$ is $H=\bigoplus_{k=1}^n \mathbf{H}_\mathbf{g}(k)$. Since $\varphi$ is a homomorphism, it pushes forward the Haar-measure to the Haar-measure. Thus, if $\mathbf{B}_n$ is a Haar-uniform element of $\mathcal{B}_{n,w_n}$, then $\mathbf{B}_n\tilde{\mathbf{g}}$ is a uniform random element of $H$, so 
\begin{equation}\label{Pker}\mathbb{P}(\mathbf{B}_n\tilde{\mathbf{g}}=0)=\prod_{k=1}^n |\mathbf{H}
    _{\mathbf{g}}(k)|^{-1}.\end{equation}
Combining this with Lemma~\ref{lemma6}, the statement follows.
\end{proof}

In Section~\ref{secmomentcalc}, we will prove the following theorem.
\begin{theorem}\label{thmmoment}
Assume that $\lim_{n\to\infty} \left(w_n-\log_p(n)\right)=+\infty$. Then
\[\lim_{n\to\infty} \sum_{\substack{\mathbf{g}\in G^{[1,n]}_{{0}_n}\\\langle\mathbf{g}\rangle=G}} \prod_{k=1}^n |\mathbf{H}
    _{\mathbf{g}}(k)|^{-1}=1.\]
\end{theorem}

The first part of Theorem~\ref{thmmain} follows by combining Theorem~\ref{thmWood}, Lemma~\ref{lemmamoment} and Theorem~\ref{thmmoment}. Theorem~\ref{thmrank} follows from Theorem~\ref{thmmain} and the following lemma.

\begin{lemma}\label{CLrank}
Let $M_n$ be a sequence of $n\times n$ random matrices over $\mathbb{Z}_p$, and let $\overline{M}_n$ be the mod $p$ reduction of $M_n$. Assume that for every finite abelian $p$-group $G$, we have
\[\lim_{n\to\infty} \mathbb{P}(\cok(M_n)\cong G)=\frac{1}{|\Aut(G)|}\prod_{j=1}^{\infty}\left(1-p^{-j}\right).\]

Then for any nonnegative integer $k$, we have
\[\lim_{n\to\infty} \mathbb{P}(\dim \ker \overline{M}_n=k)=p^{-k^2} \prod_{j=1}^{k} \left(1-p^{-j}\right)^{-2} \prod_{j=1}^{\infty}\left(1-p^{-j}\right).\]

Moreover,
\[\lim_{n\to\infty} \mathbb{P}(\dim \ker \overline{M}_n\ge k)\le Cp^{-k^2}\]
for some constant $C$ depending only on $p$.
\end{lemma}
\begin{proof}
See \cite[Theorem 6.3]{cohen2006heuristics} and \cite[Section 5]{meszaros20242}.
\end{proof}


\section{Calculation of moments}\label{secmomentcalc}

Let $t=|\Sg(G)|$, and let $G_1,G_2,\dots,G_t$ be a list of all the subgroups of $G$ such that \break $|G_1|\ge |G_2|\ge\cdots \ge |G_t|$. Clearly, $G_1=G$ and $G_t=\{0\}$.

For $\mathbf{g}\in G^{[c,d]}_\mathbf{b}$, we define
\[\ell(\mathbf{g})=\max\{i\,:\,\mathbf{H}_{\mathbf{g}}(k)=G_i\text{ for some }k\in [c,d]\}.\]

We also define
\begin{align*}G^{[c,d]}_\mathbf{b}[i]&=\{\mathbf{g}\in G^{[c,d]}_\mathbf{b}\,:\,\ell(\mathbf{g})=i\},\text{ and }\\
G^{[c,d]}_\mathbf{b}[\le i]&=\cup_{j=1}^i G^{[c,d]}_\mathbf{b}[j].\end{align*}

Let $n, w$ be positive integers. We set $q=tn p^{-w}$. For $1\le i\le t$, we define
\[\alpha_i=\alpha_i(q)=\begin{cases}
1&\text{for $i=1$,}\\
(q\beta_i+2t)^2\exp(q^2\beta_i)&\text{otherwise,}
\end{cases}\]
where $\beta_i=\sum_{j=1}^{i-1} \alpha_j$.

\begin{lemma}\label{lemmaalpha}
For all $[c,d]$ such that $d-c+1\le n$, $\mathbf{b}:[c-w,c-1]\cup [d+1,d+w]\to G$ and $1\le i\le t$, we have
\[\sum_{\mathbf{g}\in G^{[c,d]}_\mathbf{b}[i]} \prod_{k=c}^d |\mathbf{H}_\mathbf{g}(k)|^{-1}\le \alpha_i.\]

\end{lemma}
\begin{proof}
We prove by induction on $i$. First consider the case $i=1$. By definition, we have $\mathbf{H}_g(k)=G$ for all $k\in [c,d]$ and $\mathbf{g}\in G^{[c,d]}_\mathbf{b}[1]$. Thus,
\[\sum_{\mathbf{g}\in G^{[c,d]}_\mathbf{b}[1]} \prod_{k=c}^d |\mathbf{H}_\mathbf{g}(k)|^{-1}=\frac{\left| G^{[c,d]}_\mathbf{b}[1]\right|}{|G|^{d-c+1}}\le 1.\]

Now assume that $i>1$. We split the sum
\[\sum_{\mathbf{g}\in G^{[c,d]}_\mathbf{b}[i]}\, \prod_{k=c}^d |\mathbf{H}_\mathbf{g}(k)|^{-1}\]
according to the value of 
\[\mathbb{W}(\mathbf{g})=\{k\in [c,d]\,:\,\mathbf{H}_\mathbf{g}(k)=G_i\}.\]

Let  $\mathcal{W}_i$ be the set of all possible values of $\mathbb{W}(\mathbf{g})$, that is,
\[\mathcal{W}_i=\{\mathbb{W}(\mathbf{g})\,:\, \mathbf{g}\in G^{[c,d]}_\mathbf{b}[i]\}.\]

As $W$ goes through $\mathcal{W}_i$, the preimages 
\[\mathbb{W}^{-1}(W)=\{\mathbf{g}\in G^{[c,d]}_\mathbf{b}[i]\,:\,\mathbb{W}(\mathbf{g})=W\}\]
form a partition of $G^{[c,d]}_\mathbf{b}[i]$. Thus,
\[\sum_{\mathbf{g}\in G^{[c,d]}_\mathbf{b}[i]}\, \prod_{k=c}^d |\mathbf{H}_\mathbf{g}(k)|^{-1}=\sum_{W\in \mathcal{W}_i}\,\sum_{\mathbf{g}\in \mathbb{W}^{-1}(W)}\, \prod_{k=c}^d |\mathbf{H}_\mathbf{g}(k)|^{-1}.\]

Elements of $\mathcal{W}_i$ have a special property given by the following claim. 
\begin{claim}\label{cl13}
Let $x,y\in [c,d]$ such that $x\le y\le x+2w+1$, and let $\mathbf{g}\in G^{[c,d]}_\mathbf{b}[i]$. If $x,y\in \mathbb{W}(\mathbf{g})$, then $[x,y]\subset \mathbb{W}(\mathbf{g})$.

\end{claim}
\begin{proof}
Let $z\in [x,y]$, then $[z-w,z+w]\subset [x-w,x+w]\cup [y-w,y+w]$. Thus, $\mathbf{H}_\mathbf{g}(z)$ is a subgroup of the subgroup generated by $\mathbf{H}_\mathbf{g}(x)$ and $\mathbf{H}_\mathbf{g}(y)$. By our assumptions, $\mathbf{H}_\mathbf{g}(x)=\mathbf{H}_\mathbf{g}(y)=G_i$. Thus, $\mathbf{H}_\mathbf{g}(z)$ is a subgroup of $G_i$. Since $\ell(\mathbf{g})=i$, we see that $|\mathbf{H}_\mathbf{g}(z)|\ge |G_i|$. Therefore, it follows that $\mathbf{H}_\mathbf{g}(z)=G_i$.  
\end{proof}

Claim~\ref{cl13} gives that any $W\in \mathcal{W}_i$ can be uniquely written as 
\[W=\cup_{j=1}^h [c_j,d_j]\]
for $h=h(W)$ subintervals of $[c,d]$ such that  \begin{equation}\label{eqsep}c_{j+1}>d_j+2w+1\text{ for all }1\le j<h.\end{equation}
In other words, the intervals $([c_j,d_j])_{j=1}^h$ are well-separated in the sense that any gap between them has at least $2w+1$ elements.

Let \[L_j=[c_j-w,c_j-1]\cap [c,d]\quad\text{ and }\quad R_j=[d_j+1,d_j+w]\cap [c,d],\]
that is, $L_j$ is the left outer boundary of the interval $[c_j,d_j]$ and $R_j$ is the right outer boundary of  the interval $[c_j,d_j]$. 

Note that we have $L_j=[c_j-w,c_j-1]$ for $1<j\le h$ and $L_1=[\max(c,c_1-w),c_1-1]$. Similarly, we have $R_j=[d_j+1,d_j+w]$ for $1\le j< h$ and $R_h=[d_h+1,\min(d_h+w,d)]$. Note that these intervals are all pairwise disjoint by the well-separation property of $W$ given in~\eqref{eqsep}. 

For a $\mathbf{g}\in \mathbb{W}^{-1}(W)$, let
\[H_j^L(\mathbf{g})=\left\langle \mathbf{g}(k)\,:\,k\in L_j\right\rangle\quad\text{ and }\quad H_j^R(\mathbf{g})=\left\langle \mathbf{g}(k)\,:\,k\in R_j\right\rangle.\]
It is easy to see that $H_j^L(\mathbf{g})$ is a subgroup of $\mathbf{H}_\mathbf{g}(c_j)=G_i$.  Similarly, $H_j^R(\mathbf{g})$ is a subgroup of~$G_i$.
\begin{claim}\label{Claim12} Let $W\in \mathcal{W}_i$  
and $\mathbf{g}\in \mathbf{W}^{-1}(W)$. For all $k\in L_j$, we have $|\mathbf{H}_\mathbf{g}(k)|\ge p|H_j^L(\mathbf{g})|$. 

Similarly, for all $k\in R_j$, we have $|\mathbf{H}_\mathbf{g}(k)|\ge p|H_j^R(\mathbf{g})|$.
\end{claim}
\begin{proof}
Let $k\in L_j$. Since $\ell(\mathbf{g})=i$, we get $|\mathbf{H}_\mathbf{g}(k)|\ge |G_i|$. Since $k\notin W$, we have $\mathbf{H}_\mathbf{g}(k)\neq G_i$. Thus, $\mathbf{H}_\mathbf{g}(k)\cap G_i\neq \mathbf{H}_\mathbf{g}(k)$. Observe that $H_j^L(\mathbf{g})$ is a subgroup of $\mathbf{H}_\mathbf{g}(k)\cap G_i=\mathbf{H}_\mathbf{g}(k)\cap \mathbf{H}_\mathbf{g}(c_j)$, because $L_j$ is contained in both $[k-w,k+w]$  and $[c_j-w,c_j+w]$.  
 
 Therefore, $p|H_j^L(\mathbf{g})|\le p |\mathbf{H}_\mathbf{g}(k)\cap G_i|\le |\mathbf{H}_\mathbf{g}(k)|$. The second statement follows similarly.
\end{proof}

For $W\in\mathcal{W}_i$, let $h=h(W)$. We would like to split the sum \[\sum_{\mathbf{g}\in \mathbb{W}^{-1}(W)}\, \prod_{k=c}^d |\mathbf{H}_\mathbf{g}(k)|^{-1}\]
according to the list of subgroups $H_1^L(\mathbf{g}),H_1^R(\mathbf{g}),\dots,H_h^L(\mathbf{g}),H_h^R(\mathbf{g})$. To this end, given a list $F^L_1,F^R_1,\dots, F_h^L,F_h^R$ of subgroups of $G_i$, we define
\[\mathcal{A}(W,F^L_1,F^R_1,\dots, F_h^L,F_h^R)=\{\mathbf{g}\in \mathbb{W}^{-1}(W)\,:\,H_j^L(\mathbf{g})=F^L_j, H_j^R(\mathbf{g})=F_j^R\text{ for all }1\le j\le h\}. \]

Let $B=B(W)=W\cup [c-w,c-1]\cup [d+1,d+w]\cup \cup_{j=1}^h (L_j\cup R_j)$, that is, $B$ is the $w$-neighborhood of $W$ together with the outer boundary $[c-w,c-1]\cup [d+1,d+w]$ of the interval $[c,d]$. 

As a next step, we would like to sort the vectors in $\mathcal{A}(W,F^L_1,F^R_1,\dots, F_h^L,F_h^R)$ according to their restriction to the set $B$. First, we define the set of all possible restrictions as 
\[\mathcal{B}(W,F^L_1,F^R_1,\dots, F_h^L,F_h^R)=\{\mathbf{g}\restriction B\,:\,\mathbf{g}\in \mathcal{A}(W,F^L_1,F^R_1,\dots, F_h^L,F_h^R)\}\subset G^B. \]

Then as $\mathbf{f}$ goes through $\mathcal{B}(W,F^L_1,F^R_1,\dots, F_h^L,F_h^R)$, the sets
\[\mathcal{A}(W,\mathbf{f})=\{\mathbf{g}\in \mathcal{A}(W,F^L_1,F^R_1,\dots, F_h^L,F_h^R)\,:\,\mathbf{g}\restriction B=\mathbf{f}\}\]
form a partition of $\mathcal{A}(W,F^L_1,F^R_1,\dots, F_h^L,F_h^R)$.

We can bound the number of parts as follows: If $\mathbf{f}\in \mathcal{B}(W,F^L_1,F^R_1,\dots, F_h^L,F_h^R)$, then $\mathbf{f}(k)\in G_i$ for all $k\in W$, $\mathbf{f}(k)\in F_j^L$ for all $k\in L_j$, $\mathbf{f}(k)\in F_j^R$ for all $k\in R_j$ and $\mathbf{f}(k)=\mathbf{b}(k)$ for all $k\in [c-w,c-1]\cup [d+1,d+w]$. Thus,
\begin{equation}\label{BWSize}|\mathcal{B}(W,F^L_1,F^R_1,\dots, F_h^L,F_h^R)|\le |G_i|^{|W|}\prod_{j=1}^h \left(|F_j^L|^{|L_j|}\cdot |F_j^R|^{|R_j|}\right).\end{equation}

Next, we bound how much the vectors in each part $\mathcal{A}(W,\mathbf{f})$ contribute to the sum in the statement of the lemma. 

\begin{claim}\label{Claim13}
    Assume that $c_1> c+w$ and $d_h< d-w$. Then for any $\mathbf{f}\in \mathcal{B}(W,F^L_1,F^R_1,\dots, F_h^L,F_h^R)$, we have
    \[\sum_{\mathbf{g}\in \mathcal{A}(W,\mathbf{f})} \prod_{k=c}^d |\mathbf{H}
    _{\mathbf{g}}(k)|^{-1}\le |G_i|^{-|W|} \beta_i^{h+1}p^{-2hw}\prod_{j=1}^h (|F_j^L||F_j^R|)^{-w}.\]
\end{claim}
\begin{proof}
Let $[\overline{c}_0,\overline{d}_0]=[c,c_1-w-1]$, $[\overline{c}_h,\overline{d}_h]=[d_h+w+1,d]$, and for $0<j<h$, let $[\overline{c}_j,\overline{d}_j]=[d_j+w+1,c_{j+1}-w-1]$. 

For any $\mathbf{g}\in \mathcal{A}(W,\mathbf{f})$, we have
\begin{align*}\prod_{k=c}^d &|\mathbf{H}
    _{\mathbf{g}}(k)|^{-1}\\&=\left( \prod_{k\in W} |\mathbf{H}
    _{\mathbf{g}}(k)|^{-1}\right)\left(\prod_{j=1}^h \prod_{k\in L_j} |\mathbf{H}
    _{\mathbf{g}}(k)|^{-1}\right)\left(\prod_{j=1}^h \prod_{k\in R_j} |\mathbf{H}
    _{\mathbf{g}}(k)|^{-1}\right)\left(\prod_{j=0}^h \prod_{k\in [\overline{c}_j,\overline{d}_j]} |\mathbf{H}
    _{\mathbf{g}}(k)|^{-1}\right)\\&\le |G_i|^{-|W|} 
    \left(\prod_{j=1}^h p^{-w}|F_j^L|^{-w} \right)\left(\prod_{j=1}^h p^{-w}|F_j^R|^{-w} \right)\left(\prod_{j=0}^h \prod_{k\in [\overline{c}_j,\overline{d}_j]} |\mathbf{H}
    _{\mathbf{g}}(k)|^{-1}\right)\end{align*}
where we used Claim~\ref{Claim12}.

 Let $\mathbf{b}_j=\mathbf{f}\restriction ([\overline{c}_j-w,\overline{c}_j-1]\cup [\overline{d}_j+1,\overline{d}_j+w])$. For any $\mathbf{g}\in \mathcal{A}(W,\mathbf{f})$ and $0\le j\le h$, we have $\mathbf{g}\restriction [\overline{c}_j-w,\overline{d}_j+w]\in G^{[\overline{c}_j,\overline{d}_j]}_{\mathbf{b}_j}[\le i-1]$. Therefore,
 \begin{align*}
\sum_{\mathbf{g}\in \mathcal{A}(W,\mathbf{f})} &\prod_{k=c}^d |\mathbf{H}
    _{\mathbf{g}}(k)|^{-1}\\&\le  |G_i|^{-|W|} 
    \left(\prod_{j=1}^h p^{-w}|F_j^L|^{-w} \right)\left(\prod_{j=1}^h p^{-w}|F_j^R|^{-w} \right)\sum_{\mathbf{g}\in \mathcal{A}(W,\mathbf{f})}\prod_{j=0}^h \prod_{k\in [\overline{c}_j,\overline{d}_j]} |\mathbf{H}
    _{\mathbf{g}}(k)|^{-1}\\&\le |G_i|^{-|W|} p^{-2wh} \left(\prod_{j=1}^h (|F_j^L||F_j^R|)^{-w}\right)\left(\prod_{j=0}^h\sum_{\mathbf{g}\in G^{[\overline{c}_j,\overline{d}_j]}_{\mathbf{b}_j}[\le i-1] }\prod_{k=\overline{c}_j}^{\overline{d}_j} |\mathbf{H}
    _{\mathbf{g}}(k)|^{-1}\right) \\&\le |G_i|^{-|W|} p^{-2wh} \left(\prod_{j=1}^h (|F_j^L||F_j^R|)^{-w}\right)  \left(\sum_{j=1}^{i-1}\alpha_j\right)^{h+1},
\end{align*}
where in the last step we used the induction hypothesis. Recalling that $\beta_i=\sum_{j=1}^{i-1}\alpha_j$, the statement follows.
\end{proof}

\begin{claim}\label{Claim14}
If $c_1>c+w$ and $d_h<d-w$, then
\[\sum_{\mathbf{g}\in \mathbb{W}^{-1}(W)} \prod_{k=c}^d |\mathbf{H}
    _{\mathbf{g}}(k)|^{-1}\le p^{-2hw} t^{2h}\beta_i^{h+1}.\]  
\end{claim}
\begin{proof}

Combining Claim~\ref{Claim13} with \eqref{BWSize}, we obtain that for any choice of $F^L_1,F^R_1,\dots, F_h^L,F_h^R$, we have

\[\sum_{\mathbf{g}\in \mathcal{A}(W,F^L_1,F^R_1,\dots, F_h^L,F_h^R)} \prod_{k=c}^d |\mathbf{H}
    _{\mathbf{g}}(k)|^{-1}\le p^{-2hw} \beta_i^{h+1}.\]
    Summing these over all possible choices of $F^L_1,F^R_1,\dots, F_h^L,F_h^R$, we get the statement.
\end{proof}

With a similar proof, one gets:
\begin{claim}
If $c_1\le c+w$ and $d_h<d-w$, then
\[\sum_{\mathbf{g}\in \mathbb{W}^{-1}(W)} \prod_{k=c}^d |\mathbf{H}
    _{\mathbf{g}}(k)|^{-1}\le p^{-(2h-1)w-(c_1-c)} t^{2h}\beta_i^{h}.\]
If $c_1>c+w$ and $d_h\ge d-w$, then
\[\sum_{\mathbf{g}\in \mathbb{W}^{-1}(W)} \prod_{k=c}^d |\mathbf{H}
    _{\mathbf{g}}(k)|^{-1}\le p^{-(2h-1)w-(d-d_h)} t^{2h}\beta_i^{h}.\]
If $c_1\le c+w$ and $d_h\ge d-w$, then
\[\sum_{\mathbf{g}\in \mathbb{W}^{-1}(W)} \prod_{k=c}^d |\mathbf{H}
    _{\mathbf{g}}(k)|^{-1}\le p^{-(2h-2)w-(d-d_h)-(c_1-c)} t^{2h}\beta_i^{h-1}.\]
\end{claim}

Assuming that $h(W)=h$, we have at most ${{d-c+1}\choose{h}}\le \frac{n^h}{h!}$ choices for $c_1,c_2,\dots,c_h$. The same is true for $d_1,d_2,\dots,d_h$. Thus,
\begin{equation}\label{Wibound}|\{W\in \mathcal{W}_i\,:\, h(W)=h\}|\le \left(\frac{n^h}{h!}\right)^2\le \frac{n^{2h}}{(h-1)!}.\end{equation}
Similarly, for any $x,y\in [c,d]$, we have
\begin{align*}|\{W\in \mathcal{W}_i\,:\, h(W)=h, c_1(W)=x\}|&\le \frac{n^{2h-1}}{(h-1)!},\\
|\{W\in \mathcal{W}_i\,:\, h(W)=h, d_h(W)=y\}|&\le \frac{n^{2h-1}}{(h-1)!},\\
|\{W\in \mathcal{W}_i\,:\, h(W)=h, c_1(W)=x,d_h(W)=y\}|&\le \frac{n^{2h-2}}{(h-1)!}.
\end{align*}

We have
\begin{align}\label{sump1}\sum_{\substack{W\in\mathcal{W}_i,\\ c_1(W)>c+w,\\d_{h(W)}(W)<d-w}}&\sum_{\mathbf{g}\in \mathbb{W}^{-1}(W)} \prod_{k=c}^d |\mathbf{H}
    _{\mathbf{g}}(k)|^{-1}\\&=\sum_{h=1}^\infty \sum_{\substack{W\in\mathcal{W}_i,\\h(W)=h,\\c_1(W)>c+w,\\d_{h}(W)<d-w}}\sum_{\mathbf{g}\in \mathbb{W}^{-1}(W)} \prod_{k=c}^d |\mathbf{H}_{\mathbf{g}}(k)|^{-1}\nonumber\\&\le \sum_{h=1}^\infty \frac{n^{2h}}{(h-1)!}
      p^{-2hw} t^{2h}\beta_i^{h+1}\nonumber\\&=q^2\beta_i^2 \sum_{h=1}^\infty \frac{\left(q^2\beta_i\right)^{h-1}}{(h-1)!}\nonumber\\&=q^2\beta_i^2\exp(q^2\beta_i),\nonumber
\end{align}
where at the first inequality, we used Claim~\ref{Claim14} and \eqref{Wibound}.

We also have
\begin{align}\label{sump2}\sum_{\substack{W\in\mathcal{W}_i,\\ c_1(W)\le c+w,\\d_{h(W)}(W)<d-w}}&\sum_{\mathbf{g}\in \mathbb{W}^{-1}(W)} \prod_{k=c}^d |\mathbf{H}
    _{\mathbf{g}}(k)|^{-1}\\&=\sum_{h=1}^\infty \sum_{u=0}^w\sum_{\substack{W\in\mathcal{W}_i\\ h(W)=h,\\c_1(W)=c+u,\\d_{h}(W)<d-w}}\sum_{\mathbf{g}\in \mathbb{W}^{-1}(W)} \prod_{k=c}^d |\mathbf{H}_{\mathbf{g}}(k)|^{-1}\nonumber\\&\le \sum_{u=0}^{w}\sum_{h=1}^\infty \frac{n^{2h-1}}{(h-1)!}
      p^{-(2h-1)w-u} t^{2h}\beta_i^{h}\nonumber\\&=\left(\sum_{u=0}^w p^{-u}\right)tq\beta_i \sum_{h=1}^\infty \frac{\left(q^2\beta_i\right)^{h-1}}{(h-1)!}\nonumber\\&\le 2tq\beta_i\exp(q^2\beta_i).\nonumber
\end{align}

Similarly,
\begin{equation}\label{sump3}
    \sum_{\substack{W\in\mathcal{W}_i,\\ c_1(W)> c+w,\\d_{h(W)}(W)\ge d-w}}\sum_{\mathbf{g}\in \mathbb{W}^{-1}(W)} \prod_{k=c}^d |\mathbf{H}
    _{\mathbf{g}}(k)|^{-1}\le 2tq\beta_i\exp(q^2\beta_i).
\end{equation}

Finally,
\begin{align}\label{sump4}\sum_{\substack{W\in\mathcal{W}_i,\\ c_1(W)\le c+w,\\d_{h(W)}(W)\ge d-w}}&\sum_{\mathbf{g}\in \mathbb{W}^{-1}(W)} \prod_{k=c}^d |\mathbf{H}
    _{\mathbf{g}}(k)|^{-1}\\&=\sum_{h=1}^\infty \sum_{u=0}^w \sum_{v=0}^w\sum_{\substack{W\in\mathcal{W}_i,\\ h(W)=h,\\c_1(W)=c+u,d_{h}(W)=d-v}}\sum_{\mathbf{g}\in \mathbb{W}^{-1}(W)} \prod_{k=c}^d |\mathbf{H}_{\mathbf{g}}(k)|^{-1}\nonumber\\&\le \sum_{u=0}^{w}\sum_{v=0}^{w}\sum_{h=1}^\infty \frac{n^{2h-2}}{(h-1)!}
      p^{-(2h-2)w-u-v} t^{2h}\beta_i^{h-1}\nonumber\\&=\left(\sum_{u=0}^w p^{-u}\right)^2t^2\sum_{h=1}^\infty \frac{\left(q^2\beta_i\right)^{h-1}}{(h-1)!}\nonumber\\&\le 4t^2\exp(q^2\beta_i).\nonumber
\end{align}

For later use, we record that almost the same argument gives that
\begin{equation}\label{sump4b}\sum_{\substack{W\in\mathcal{W}_i\\h(W)>1\\ c_1(W)\le c+w,\\d_{h(W)}(W)\ge d-w}}\sum_{\mathbf{g}\in \mathbb{W}^{-1}(W)} \prod_{k=c}^d |\mathbf{H}
    _{\mathbf{g}}(k)|^{-1}\le 4t^2\left(\exp(q^2\beta_i)-1\right).\end{equation}
Combining \eqref{sump1},\eqref{sump2},\eqref{sump3} and \eqref{sump4}, we obtain Lemma~\ref{lemmaalpha}.
\end{proof}

Let
\[\mathcal{W}^*_i=\mathcal{W}_i\cap \{[c_1,d_1]\,:\,c\le c_1\le c+w, d-w\le d_1\le d\}.\]

Combining  \eqref{sump1},\eqref{sump2},\eqref{sump3} and \eqref{sump4b}, we see that
\begin{equation}\label{sumpc}\sum_{W\in\mathcal{W}_i\setminus \mathcal{W}^*_i}\sum_{\mathbf{g}\in \mathbb{W}^{-1}(W)} \prod_{k=c}^d |\mathbf{H}
    _{\mathbf{g}}(k)|^{-1}\le q\beta_i(4t+q\beta_i)\exp(q^2\beta_i)+4t^2\left(\exp(q^2\beta_i)-1\right).\end{equation}


\begin{lemma}\label{lemmamomentupper}
Assume that $\lim_{n\to\infty}\left(w_n-\log_p(n)\right)=+\infty$,
then 
\[\limsup_{n\to\infty} \sum_{\substack{\mathbf{g}\in G^{[1,n]}_{{0}_n}\\\langle\mathbf{g}\rangle=G}} \prod_{k=1}^n |\mathbf{H}
    _{\mathbf{g}}(k)|^{-1}\le 1.\]

\end{lemma}
\begin{proof}
Let $q_n=tnp^{-w_n}$. Clearly, $\lim_{n\to\infty} q_n=0$. Since $\beta_i(q)$ is monotone increasing in $q$, there is a $C$ such that $\beta_i(q_n)<C$ for all $n$ and $2\le i\le t$. It follows from \eqref{sumpc} that
\begin{align*}\sum_{\substack{\mathbf{g}\in G^{[1,n]}_{{0}_n}\\\mathbb{W}(\mathbf{g})\notin \mathcal{W}_{\ell(\mathbf{g})}^*}} \prod_{k=1}^n |\mathbf{H}
    _{\mathbf{g}}(k)|^{-1}&\le \sum_{i=2}^t \left( q_n\beta_i(q_n)(4t+q_n\beta_i(q_n))\exp(q_n^2\beta_i(q_n))+4t^2\left(\exp(q_n^2\beta_i(q_n))-1\right)\right)\\&\le (t-1)\left( q_nC(4t+q_nC)\exp(q_n^2C)+4t^2\left(\exp(q_n^2C)-1\right)\right). \end{align*}
Here the right hand side converges to $0$. Therefore, 
\[\limsup_{n\to\infty} \sum_{\substack{\mathbf{g}\in G^{[1,n]}_{{0}_n}\\\langle\mathbf{g}\rangle=G}} \prod_{k=1}^n |\mathbf{H}
    _{\mathbf{g}}(k)|^{-1}= \limsup_{n\to\infty} \sum_{\substack{\mathbf{g}\in G^{[1,n]}_{{0}_n}\\\langle\mathbf{g}\rangle=G,\mathbb{W}(\mathbf{g})\in \mathcal{W}^*_{\ell(\mathbf{g})}}} \prod_{k=1}^n |\mathbf{H}
    _{\mathbf{g}}(k)|^{-1}.\]

Observe that for any $\mathbf{g}\in G^{[1,n]}_{{0}_n}$, we have $\mathbf{H}_{\mathbf{g}}(1)\subset \mathbf{H}_{\mathbf{g}}(2)\subset \cdots \subset \mathbf{H}_{\mathbf{g}}(\min(n,w_n+1))$ and $\mathbf{H}_{\mathbf{g}}(n)\subset \mathbf{H}_{\mathbf{g}}(n-1)\subset \cdots \subset \mathbf{H}_{\mathbf{g}}(\max(1,n-w_n))$. Thus, if $\mathbb{W}(\mathbf{g})\in \mathcal{W}^*_{\ell(\mathbf{g})}$, then it follows that $\mathbb{W}(\mathbf{g})=[1,n]$. Thus, $\mathbf{H}_\mathbf{g}(k)=G_{\ell(\mathbf{g})}$ for all $k\in [1,n]$. If we further require that $\langle \mathbf{g}\rangle=G$, then it follows that $\mathbf{H}_\mathbf{g}(k)=G$ for all $k\in [1,n]$. Therefore,

\begin{equation}\label{onlytypical}
    \lim_{n\to\infty} \left(\sum_{\substack{\mathbf{g}\in G^{[1,n]}_{{0}_n}\\\langle\mathbf{g}\rangle=G}} \prod_{k=1}^n |\mathbf{H}
    _{\mathbf{g}}(k)|^{-1}\quad \mathlarger{-}\sum_{\substack{\mathbf{g}\in G^{[1,n]}_{{0}_n}\\\mathbf{H}_\mathbf{g}(k)=G\text{ for all }k\in[1,n]}} \prod_{k=1}^n |\mathbf{H}
    _{\mathbf{g}}(k)|^{-1}\right)=0.
\end{equation}

Thus,
\[\limsup_{n\to\infty} \sum_{\substack{\mathbf{g}\in G^{[1,n]}_{{0}_n}\\\langle\mathbf{g}\rangle=G}} \prod_{k=1}^n |\mathbf{H}
    _{\mathbf{g}}(k)|^{-1}= \limsup_{n\to\infty} \sum_{\substack{\mathbf{g}\in G^{[1,n]}_{{0}_n}\\\mathbf{H}_\mathbf{g}(k)=G\text{ for all }k\in[1,n]}} \prod_{k=1}^n |\mathbf{H}
    _{\mathbf{g}}(k)|^{-1}\le |G|^n|G|^{-n}=1.
\]

\end{proof}

Observe that \begin{align}\label{sq0}\liminf_{n\to\infty} \sum_{\substack{\mathbf{g}\in G^{[1,n]}_{{0}_n}\\\langle\mathbf{g}\rangle=G}} \prod_{k=1}^n |\mathbf{H}
    _{\mathbf{g}}(k)|^{-1}&\ge \liminf_{n\to\infty}(|G|^n-\sum_{G\neq G'\in \Sg(G)} |G'|^n)|G|^{-n}\\&\ge \liminf_{n\to\infty} (1-|\Sg(G)|p^{-n})=1.\nonumber\end{align}

Combining this with Lemma~\ref{lemmamomentupper}, we obtain Theorem~\ref{thmmoment}.

We finish this section by proving Theorem~\ref{thmdeloc}.

We set $G=\mathbb{F}_p$. Given a $\tilde{\mathbf{g}}\in G^{[1,n]}$, let  $\mathbf{g}\in G^{[1,n]}_{{0}_n}$  be obtained by appending $\tilde{\mathbf{g}}$ with zeros.

Let \[\mathcal{E}_n'=\{\tilde{\mathbf{g}}\in \mathcal{E}_n\,:\,\mathbf{H}_\mathbf{g}(k)=\mathbb{F}_p\text{ for all }k\in[1,n]\}.\]

Since any nonzero element of the group $\mathbb{F}_p$ generates $\mathbb{F}_p$, $\langle \mathbf{g}\rangle=\mathbb{F}_p$ for all nonzero $\mathbf{g}$. Thus, under the assumption that $\lim_{n\to\infty} \left(w_n-\log_p(n)\right)=\infty$, it follows from \eqref{onlytypical} that
\[\limsup_{n\to\infty} \sum_{\tilde{\mathbf{g}}\in \mathcal{E}_n} \prod_{k=1}^n |\mathbf{H}
    _{\mathbf{g}}(k)|^{-1}= \limsup_{n\to\infty} \sum_{\tilde{\mathbf{g}}\in \mathcal{E}_n'}\prod_{k=1}^n |\mathbf{H}
    _{\mathbf{g}}(k)|^{-1}.\]

Combining this with the union bound, \eqref{Pker} and the definition of $\mathcal{E}_n'$, we see that
\begin{align*}
\limsup_{n\to\infty}\mathbb{P}(\ker \overline{\mathbf{B}}_n\cap \mathcal{E}_n\neq\emptyset)&\le\limsup_{n\to\infty}\sum_{\tilde{\mathbf{g}}\in \mathcal{E}_n} \mathbb{P}(\mathbf{B}_n\tilde{\mathbf{g}}=0)\\&=\limsup_{n\to\infty} \sum_{\tilde{\mathbf{g}}\in \mathcal{E}_n} \prod_{k=1}^n |\mathbf{H}
    _{\mathbf{g}}(k)|^{-1}\\&= \limsup_{n\to\infty} \sum_{\tilde{\mathbf{g}}\in \mathcal{E}_n'}\prod_{k=1}^n |\mathbf{H}
    _{\mathbf{g}}(k)|^{-1}\\&=\limsup_{n\to\infty}|\mathcal{E}_n'|\cdot p^{-n}=0.
\end{align*}
Thus, Theorem~\ref{thmdeloc} follows.
\section{Localization}\label{secloc}

Considering $\mathbb{F}_p$ as a finite abelian $p$-group, for all $v\in \mathbb{F}_p^n$, one can define $\mathbf{B}_n v$ as in Section~\ref{sec2}. Observe that with this definition, we have $\mathbf{B}_n v=\overline{\mathbf{B}}_n v$. 

If $w_n-\log_p(n)$ does not converge to $+\infty$, then there is $\varepsilon\in (0,1)$ such that $np^{-w_n}\ge\varepsilon$ for infinitely many $n$. Thus, Theorem~\ref{thmloc} follows from the lemma below.

\begin{lemma}\label{loclemma}
Let $\varepsilon>0$ and $\lC\ge \varepsilon $ be an integer. Choose $w\ge 1$ and $n> 2\lC$ such that $np^{-w}\ge \varepsilon$. Let $\mathbf{B}_n$ be a Haar-uniform element of $\mathcal{B}_{n,w}$. Set $n_0= \left\lfloor\frac{n}\lC \right\rfloor$. 

Let
\[V_i=\left\{v\in \mathbb{F}_p^n\,:\,\emptyset\neq \supp v\subset \left[(i-1)n_0+1\,,\,i n_0\right]\right\}.\]

Let us consider the random variables
\[X_i=|\{v\in V_i\,:\, \mathbf{B}_nv=0 \}|.\]

Then these random variables are independent and
\[\mathbb{P}(X_i\ge 1)\ge \left(C\frac{\varepsilon}\lC \right)^{4}\text{ for all }i=1,2,\dots,\lC\text{ and }\]
\[\mathbb{P}(X_i\ge 1\text{ for all }i=1,2,\dots,\lC )\ge \left(C\frac{\varepsilon}\lC \right)^{4\lC }\]
for some positive constant $C$ only depending on $p$.

\end{lemma}
\begin{proof}
Let us set $n_1=\min\left(n_0,p^{w}\right)$.
Using the assumptions $n>2L$ and $np^{-w}\ge \varepsilon$, we have 
\begin{equation}\label{qbound}\frac{n_1}{p^w}= \min\left(\left\lfloor\frac{n}{\lC}\right\rfloor\frac{1}{p^w},1\right)\ge \min\left(\frac{n}{2\lC p^w},1\right)\ge \min\left(\frac{\varepsilon}{2\lC },1\right)=\frac{\varepsilon}{2\lC }, \end{equation}
where in the last step we used $\lC\ge\varepsilon$.

Let $(i-1)n_0+1\le c<d\le (i-1) n_0+n_1$, and let
\[U=\{v\in \mathbb{F}_p^n\,:\,\{c,d\}\subset \supp v\subset [c,d]\}.\]

Note that
\begin{equation}\label{Ubound}|U|=(p-1)^2p^{d-c-1}\ge p^{d-c-1}.\end{equation}
For all $v\in U$, the distribution of $\mathbf{B}_nv$ is uniform on some subgroup of 
\[U'=\{v\in \mathbb{F}_p^n\,:\, \supp v\subset [c-w,d+w]\}.\]

Thus,
\[\mathbb{P}(\mathbf{B}_nv=0)\ge p^{-(d-c+1+2w)}.\]
Combining this with \eqref{Ubound}, we obtain
\begin{equation}\label{EUb}\mathbb{E}|\{v\in U\,:\,\mathbf{B}_nv=0\}|\ge p^{-2w-2}.\end{equation}

We define
\[V_i'=\left\{v\in \mathbb{F}_p^n\,:\,\emptyset\neq \supp v\subset \left[(i-1)n_0+1\,,\,(i-1) n_0+n_1\right]\right\},\]
and the random variables
\[Y_i=|\{v\in V_i'\,:\, \mathbf{B}_nv=0 \}|.\]

Clearly, $V_i'\subset V_i$, and $Y_i\le X_i$.
Then using \eqref{EUb}, we get
\[\mathbb{E}Y_i\ge \sum_{(i-1)n_0+1\le c<d\le (i-1)n_0+n_1} p^{-2w-2}={{n_1}\choose{2}}p^{-2w-2}\ge \frac{n_1^2 p^{-2w}}{4p^2},\]
using the fact that $n_1>1$ so $n_1-1\ge \frac{n_1}2$.

\begin{claim}
There is a constant $\gamma$ not depending on $n$ such that
\[\mathbb{E} Y_i^2\le \gamma\quad\text{ for all }i=1,2,\dots,\lC.\]
\end{claim}
\begin{proof}

Let $\mathbf{B}'$ be the submatrix of $\mathbf{B}_n$ determined by the columns and rows with index in the interval $[a,z]=[(i-1)n_0+1\,,\,(i-1) n_0+n_1]$. Note $\mathbf{B}'$ has the same distribution as an $n_1\times n_1$ Haar-uniform band matrix with band width $w$.

We set $G=\mathbb{F}_p\times \mathbb{F}_p$. Consider a pair of vectors $v,v'\in V_i'$, and let $\mathbf{g}\in G^{n_1}=(\mathbb{F}_p\times \mathbb{F}_p)^{n_1}$ be defined by  \[\mathbf{g}_j=\left(v_{a+j-1},v'_{a+j-1}\right) \text{ for all }j=1,2,\dots,n_1.\]

Observe that on the event that  $\mathbf{B}_nv=0$ and $\mathbf{B}_nv'=0$, we have $\mathbf{B}'\mathbf{g}=0$. Therefore,
\[\mathbb{E} Y_i^2=\mathbb{E}|\{(v,v')\in (V_i')^2\,:\,\mathbf{B}_nv=0,\mathbf{B}_nv'=0 \}|\le \mathbf{E}|\{\mathbf{g}\in G^{n_1}\,:\,\mathbf{B}'\mathbf{g}=0\}|.\]

Combining Lemma~\ref{lemma6} and Lemma~\ref{lemmamoment}, the right hand side can be rewritten as
\[\mathbf{E}|\{\mathbf{g}\in G^{n_1}\,:\,\mathbf{B}'\mathbf{g}=0\}|=\sum_{\mathbf{g}\in G^{[1,n_1]}_{0_{n_1}}} \quad\prod_{k=1}^{n_1} |\mathbf{H}_{\mathbf{g}}(k)|^{-1}.\]
Next we use Lemma~\ref{lemmaalpha} to bound this expression. By the choice of $n_1$, we have $q=tn_1p^{-w}\le t$. Since $\alpha_i$ is monotone increasing in $q$, it follows from Lemma~\ref{lemmaalpha} that there is a $\gamma$ not depending on $n$ such that
\[\sum_{\mathbf{g}\in G^{[1,n_1]}_{0_{n_1}}} \quad\prod_{k=1}^{n_1} |\mathbf{H}_{\mathbf{g}}(k)|^{-1}\le \gamma.\]
Thus, the claim follows.
\end{proof}

From the Paley-Zygmund inequality,
\begin{equation}\label{paley}\mathbb{P}(X_i\ge 1)\ge\mathbb{P}(Y_i\ge 1)\ge \frac{(\mathbb{E} Y_i)^2}{\mathbb{E} Y_i^2}\ge \frac{n_1^4 p^{-4w}}{16p^4\gamma}\ge \left(C\frac{\varepsilon}\lC \right)^4,\end{equation}
where $C=\frac{1}{4p\sqrt[4]{\gamma}}$. At the last inequality we used \eqref{qbound}.

Note that for all $v\in V_i$ the event $\mathbf{B}_n v=0$ only depends on the columns of $\mathbf{B}_n$  indexed by $[(i-1)n_0+1,in_0]$. Thus, $X_1,\dots,X_\lC $ are independent. Combining this with  \eqref{paley}, we obtain
\[\mathbb{P}(X_i\ge 1\text{ for all }i=1,2,\dots,\lC )\ge  \left(C\frac{\varepsilon}\lC \right)^{4\lC }.\qedhere\]
\end{proof}

\section{Tightness}\label{sectight}
In this section, we prove Theorem~\ref{thmtight}.

First assume that $\liminf_{n\to\infty} \left(w_n-\log_p(n)\right)>-\infty$. Let $q_n=2np^{-w_n}$. Then there is a $q$ such that $q_n\le q$ for all $n$. Combining Lemma~\ref{lemmamoment} and Lemma~\ref{lemmaalpha}, we see that
\begin{align*}
\mathbb{E}|\ker \overline{\mathbf{B}}_n|&=\mathbb{E}|\Hom(\cok(\mathbf{B}_n),\mathbb{F}_p)|\nonumber\\&=\sum_{\mathbf{g}\in (\mathbb{F}_p)^{[1,n]}_{{0}_n}} \prod_{k=1}^n |\mathbf{H}
    _{\mathbf{g}}(k)|^{-1}\nonumber
    \\&\le \alpha_1(q_n)+\alpha_2(q_n)\\&\le 1+(q_n+4)^2\exp\left(q_n^2\right)\\&\le 1+(q+4)^2\exp(q^2),\nonumber
\end{align*}
using the fact that for the choice of $G=\mathbb{F}_p$, $t=2$ in Lemma~\ref{lemmaalpha}.

Thus, $\sup_n \mathbb{E}|\ker \overline{\mathbf{B}}_n|<\infty$, which by Markov's inequality implies that the sequence $|\ker \overline{\mathbf{B}}_n|$ is tight. Therefore, the sequence $\dim \ker \overline{\mathbf{B}}_n=\log_p |\ker \overline{\mathbf{B}}_n|$ is also tight.

Now, assume that $\liminf_{n\to\infty} \left(w_n-\log_p(n)\right)=-\infty$.

Let $\lC$ be a positive integer and set $\varepsilon=\lC$. There are infinitely many $n$ such that $n p^{-w_n}\ge \varepsilon$ and $n> 2\lC$. For any such $n$, Lemma~\ref{loclemma} can be applied to give that $\mathbb{P}(X_i\ge 1)\ge C^4$. Let $I_i$ be the indicator of the event that $X_i\ge 1$. By Lemma~\ref{loclemma}, $I_1,I_2,\dots,I_\lC$ are independent and
\[\mathbb{E}\sum_{i=1}^{\lC} I_i\ge \lC C^4.\]
Since $\dim\ker \overline{\mathbf{B}}_n\ge \sum_{i=1}^\lC I_i$, it follows from Hoeffding's inequality that
\[\mathbb{P}\left(\dim \ker \overline{\mathbf{B}}_n\le \frac{\lC C^4}2\right)\le \mathbb{P}\left(\sum_{i=1}^{\lC} I_i\le \frac{\lC C^4}2\right)\le \exp\left(-\frac{C^8\lC}{2}\right),\]
which clearly makes it impossible that the sequence $\dim \ker \overline{\mathbf{B}}_n$ is tight. 

 \section{Normalized dimension}\label{secnorm}
 In this section, we prove Theorem~\ref{thmnormalized}.

First assume that $\lim_{n\to\infty} w_n=\infty.$ Given $n$, let us write $[1,n]$ as the disjoint union of $h=\lceil\frac{n}{w_n}\rceil$ intervals $I_1,\dots,I_h$ such that $|I_j|=n_j\le w_n.$ Let $\mathbf{M}_j$ be the $n_j\times n_j$ submatrix of $\overline{\mathbf{B}}_n$ determined by the rows and columns indexed by $I_j$. Since $n_j\le w_n$, we see that $\mathbf{M}_j$ is a uniform random $n_j\times n_j$ matrix over $\mathbb{F}_p$. Finally, let $\mathbf{C}_n$ be an $n\times n$ matrix over $\mathbb{F}_p$ defined by
\[\mathbf{C}_n(x,y)=\begin{cases}0&\text{ for }(x,y)\in \cup_{j=1}^h I_j\times I_j,\\
\overline{\mathbf{B}}_n(x,y)&\text{ otherwise.}
\end{cases}\]
 Note that even if we condition on $\mathbf{C}_n$, $\mathbf{M}_j$ is still a uniform random $n_j\times n_j$ matrix over $\mathbb{F}_p$ and $\mathbf{M}_1,\dots,\mathbf{M}_h$ are conditionally independent. Given a vector $v\in \mathbb{F}_p^n$, let $v_{I_j}\in \mathbb{F}_p^{n_j}$ be its restriction to $I_j$.

Note that $\mathbf{B}_n v=0$ if and only if $\mathbf{M}_j v_{I_j}=-(\mathbf{C}_nv)_{I_j}$ for all $1\le j\le h$. Using the conditional independence of $\mathbf{M}_1,\dots,\mathbf{M}_h$, we obtain
\[\mathbb{P}(\mathbf{B}_n v=0\,|\,\mathbf{C}_n)=\prod_{j=1}^h \mathbb{P}(\mathbf{M}_j v_{I_j}=-(\mathbf{C}_nv)_{I_j}\,|\,\mathbf{C}_n). \]

If $v_{I_j}\neq 0$, then $\mathbf{M}_j v_{I_j}$ even conditioned on $\mathbf{C}_n$ is a uniform random element of $\mathbb{F}_p^{n_j}$. Thus, in this case
\[\mathbb{P}(\mathbf{M}_j v_{I_j}=-(\mathbf{C}_nv)_{I_j}\,|\,\mathbf{C}_n)=p^{-n_j}.\]
Therefore,
\[\mathbb{P}(\mathbf{B}_n v=0)\le \prod_{\substack{1\le j\le h\\v_{I_j}\neq 0}} p^{-n_j}.\]
So
\[\mathbb{E}|\ker \overline{\mathbf{B}}_n|=\sum_{v\in\mathbb{F}_p^n} \mathbb{P}(\mathbf{B}_n v=0)\le \prod_{j=1}^h(1+(p^{n_j}-1)p^{-n_j})\le 2^h\le 2^{n/w_n+1}.\]

Combining this with Markov's inequality,  for all $\varepsilon>0$, we obtain that
\begin{multline*}\limsup_{n\to\infty}\mathbb{P}(\dim\ker \overline{\mathbf{B}}_n\ge \varepsilon n)=\limsup_{n\to\infty}\mathbb{P}(|\ker \overline{\mathbf{B}}_n|\ge p^{\varepsilon n}) 
\\\le\limsup_{n\to\infty}\frac{\mathbb{E}|\ker \overline{\mathbf{B}}_n|}{p^{\varepsilon n}}\le\limsup_{n\to\infty}2^{n/w_n+1-\varepsilon n}=0. \end{multline*}

Thus,
 \[\frac{\dim \ker \overline{\mathbf{B}}_n}n\]
 converges to zero in probability.

Now assume that $w_n$ does not converge to infinity. Then there is a finite $w$ such that $w_n=w$ infinitely often. Consider such an $n$. Let $Z_i$ be the indicator of the event that the $i$th column of $\overline{\mathbf{B}}_n$ is the all zero vector. Observe that $\dim \ker\overline{\mathbf{B}}_n\ge \sum_{i=1}^n Z_i$. We have $\mathbb{E}{Z}_i\ge p^{-2w-1}$, so  $\mathbb{E}\sum_{i=1}^n{Z}_i\ge p^{-2w-1}n$. Clearly, $Z_1,\dots,Z_n$ are independent, thus by Hoeffding's inequality, we see that
\[\mathbb{P}\left(\dim \ker\overline{\mathbf{B}}_n\le \frac{1}2 p^{-2w-1}n\right)\le \mathbb{P}\left(\sum_{i=1}^{n} Z_i\le \frac{1}2p^{-2w-1}n\right)\le \exp\left(-\frac{1}2 p^{-4w-2}n\right),\]
which shows that \[\frac{\dim \ker \overline{\mathbf{B}}_n}n\]
 cannot converge to zero in probability.

\section{Open problems}\label{secopen}

\subsection{Behaviour at criticality}

Our tightness result given in Theorem~\ref{thmtight} points towards the following conjecture.
\begin{conjecture}\label{conjcritical}
  There is a one parameter family of distributions $(\nu_{\alpha})_{\alpha\in\mathbb{R}}$ on the set of finite abelian $p$-groups with the following property. Let $n_1<n_2<\dots$ be a sequence of positive integers. Let $\mathbf{B}_i$ be a Haar-uniform random element of $\mathcal{B}_{n_i,w_i}$. Let us assume that $\lim_{i\to\infty} \left(w_i-\log_p(n_i)\right)=\alpha$. Then for a finite abelian $p$-group $G$, we have
  \[\lim_{i\to\infty} \mathbb{P}(\cok(\mathbf{B}_i)\cong G)=\nu_\alpha(G).\]
\end{conjecture}
   Provided that this conjecture is true, $\nu_\alpha$ must have a tail which is heavier than the tail of the Cohen-Lenstra distribution as we can see from Theorem~\ref{thmloc}. Note that if the $2$-torsion of the homology of determinantal hypertrees has a limiting distribution, then this limiting distribution also has a tail which is heavier than the tail of the Cohen-Lenstra distribution \cite{meszaros20242}.

    For truncated block bidiagonal matrices, the analogue of Conjecture~\ref{conjcritical} was proved by the author on the level of coranks~\cite{meszaros2025rank}. 

\subsection{Symmetric band matrices}

\begin{problem}
    Can we prove an analogue of Theorem~\ref{thmmain} for symmetric band matrices? The Cohen-Lenstra distribution should be replaced with a modified version of the Cohen-Lenstra distribution given in \cite{clancy2015note,clancy2015cohen}.
\end{problem}

\subsection{General entry distributions}

Let $\tau$ be a probability measure on $\mathbb{Z}_p$ which is non-degenerate in the sense that $\tau(i+p\mathbb{Z}_p)<1$ for all $i=0,1,\dots,p-1.$ Consider band matrices where the non-zero entries are i.i.d. with law~$\tau$.

\begin{problem}
    Can we prove an analogue of Theorem~\ref{thmmain} for such band matrices? Is it true that the phase transition happens at $c(\tau)\log_p(n)$, where $c(\tau)$ is some function of $\tau$?
\end{problem}

For some progress on this problem, see the work of Rajagopal~\cite{rajagopal2026universality}.

\subsection{General prescription matrices}

By a prescription matrix $P$, we mean an $n\times n$ matrix, where each entry is either $0$ or $1$. The set of $P$-band matrices is defined as
\[\mathcal{B}_{P}=\{B\in M_n(\mathbb{Z}_p)\,:\,B(i,j)=0\text{ for all }i,j\text{ such that }P(i,j)=0\}.\]

Let $\mathbf{B}_P$ be a Haar uniform element of $\mathcal{B}_{P}$. 

We say a sequence $(P_n)$ of prescription matrices is weakly Cohen-Lenstra, if the distribution of $\cok(\mathbf{B}_{P_n})$ converges to the Cohen-Lenstra distribution. We say a sequence $(P_n)$ of prescription matrices is strongly Cohen-Lenstra, if $\lim_{n\to\infty} \mathbb{E}|\Sur(\cok(\mathbf{B}_{P_n}),G)|=1$ for all finite abelian $p$-groups $G$. It follows from Theorem~\ref{thmWood} that if $(P_n)$ is strongly Cohen-Lenstra, then it is weakly Cohen-Lenstra.

\begin{problem}
    Is it true that $(P_n)$ is weakly Cohen-Lenstra if and only if it is strongly Cohen-Lenstra?
\end{problem}

\begin{problem}\label{problem22}
    Try to characterize the weakly/strongly Cohen-Lenstra sequences. 
\end{problem}

Note that Theorem~\ref{thmmain} provides such a characterization in the special case when $P_n$ is an $n\times n$ matrix such that
\[P_n(i,j)=\begin{cases}1&\text{if $|i-j|\le w_n$},\\0&\text{otherwise}.\end{cases}\]

The family of $d$-dimensional band matrices, which we define next, provides an interesting special case of Problem~\ref{problem22}. Fix a dimension $d\ge 1$. For all $n$ consider a band width $w_n$, and define $P_{n,w_n}$ as a prescription matrix, where the rows and columns are indexed by $\{1,2,\dots,n\}^d$ and given two vectors $\mathbf{x},\mathbf{y}\in \{1,2,\dots,n\}^d$, we have
\[P_{n,w_n}(\mathbf{x},\mathbf{y})=\begin{cases}
    1&\text{ if }\|\mathbf{x}-\mathbf{y}\|_1\le w_n\\
    0&\text{ otherwise,} 
\end{cases}\]
where $\|\mathbf{z}\|_{1}=\sum_{i=1}^d |\mathbf{z}_i|$ is the $\ell^1$-norm on $\mathbb{R}^d$.
\begin{problem}
For which choices of band widths is $(P_{n,w_n})$ a weakly/strongly Cohen-Lenstra sequence?
\end{problem}
The same question is also interesting if we use the $\ell^\infty$-norm in place of the $\ell^1$-norm in the definition of $P_{n,w_n}$.

Note that the set of strongly Cohen-Lenstra sequences of prescription matrices is upwardly closed. More precisely, we have the following lemma.

\begin{lemma}
Let $(P_n)$ and $(Q_n)$ be two sequences of prescription matrices such that for all $n$, $P_n$ and $Q_n$ have the same sizes, and $P_n\le Q_n$ entrywise.  If $(P_n)$ is strongly Cohen-Lenstra, then $(Q_n)$ is also strongly Cohen-Lenstra.
\end{lemma}
\begin{proof}
Let $P$ be an $n\times n$ prescription matrix, and
let $G$ be a finite abelian $p$-group. For $\mathbf{g}\in G^n$, let us define $\mathbf{H}_{\mathbf{g},P}:[1,n]\to \Sg(G)$ by
\[\mathbf{H}_{\mathbf{g},P}(k)=\langle \mathbf{g}(j)\,:\, P(k,j)=1\rangle.\]

Then $\mathbf{B}_P\mathbf{g}$ is a uniform random element of $\bigoplus_{k=1}^n \mathbf{H}_{\mathbf{g},P}(k)$. Thus, as in the proof of Lemma~\ref{lemmamoment}, we see that
\[\mathbb{E}|\Sur(\cok(\mathbf{B}_P),G)|=\sum_{\substack{\mathbf{g}\in G^n\\\langle\mathbf{g}\rangle=G}}\prod_{k=1}^n |\mathbf{H}_{\mathbf{g},P}(k)|^{-1}, \]
where as before $\langle\mathbf{g}\rangle$ is the subgroup generated by the components of $\mathbf{g}.$

If $P\le Q$ entrywise, then $\mathbf{H}_{\mathbf{g},P}(k)$ is a subgroup of $\mathbf{H}_{\mathbf{g},Q}(k)$ for any $\mathbf{g}\in G^n$ and $1\le k\le n$. Thus, it follows that
\[\mathbb{E}|\Sur(\cok(\mathbf{B}_P),G)|\ge \mathbb{E}|\Sur(\cok(\mathbf{B}_Q),G)|\ge 1-|\Sg(G)|p^{-n},\]
where the last inequality follows the same way as \eqref{sq0}. The lemma follows from the squeeze theorem.
\end{proof}

The next problem asks whether Theorem~\ref{thmmain} provides us the asymptotically sparsest weakly (or strongly) Cohen-Lenstra sequence of prescription matrices.
\begin{problem}\label{problemmin}
  Let $(P_n)$ be a weakly/strongly Cohen-Lenstra sequence of prescription matrices such that $P_n$ is an $n\times n$ matrix. Is it true that 
  \[\liminf_{n\to\infty}\frac{1}{n\log_p(n)}\sum_{i=1}^n\sum_{j=1}^n P_n(i,j)\ge 2?\]
\end{problem} 

Kang, Lee and Yu~\cite{kang2026random} proved that the $\liminf$ in Problem~\ref{problemmin} is at least $1$. They also proved results which point towards the conjecture that this lower bound can be achieved by considering adjacency matrices of random regular graphs. Thus, the answer to the question in Problem~\ref{problemmin} is likely negative.   
\bibliography{references}
\bibliographystyle{plain}

\bigskip

\bigskip

\noindent Andr\'as M\'esz\'aros, \\
HUN-REN Alfr\'ed R\'enyi Institute of Mathematics, \\Budapest, Hungary,\\ {\tt meszaros@renyi.hu}
\end{document}